\documentclass[12pt,reqno]{article}

\usepackage[usenames]{color}
\usepackage{amssymb, amsmath, amsthm}
\usepackage{fullpage}

\usepackage{hyperref}

\definecolor{webgreen}{rgb}{0,.5,0}
\definecolor{webbrown}{rgb}{.6,0,0}

\begin{document}

\begin{center}
\vskip 1cm{\LARGE\bf
On variants of Conway and Conolly's\\
Meta-Fibonacci recursions\\
}
\vskip 1cm \large
Abraham Isgur \quad Mustazee Rahman\\
Department of Mathematics\\
University of Toronto\\
Toronto, Ontario M5S 2E4\\
Canada\\
\href{mailto:umarovi@gmail.com}{\tt umarovi@gmail.com}\,, \href{mailto:mustazee.rahman@mail.utoronto.ca}{\tt mustazee.rahman@mail.utoronto.ca}
\end{center}

\vskip .2 in

%July 24 10 %version

\begin{abstract}We study the recursions $A(n) = A(n-a-A^k(n-b)) + A(A^k(n-b))$ where $a \geq 0$, $b \geq 1$ are integers and the
superscript $k$ denotes a $k$-fold composition, and also the recursion $C(n) = C(n-s-C(n-1)) + C(n-s-2-C(n-3))$ where $s \geq 0$
is an integer. We prove that under suitable initial conditions the sequences $A(n)$ and $C(n)$ will be defined for all positive integers,
and be monotonic with their forward difference sequences consisting only of 0 and 1. We also show that the sequence generated by
the recursion for $A(n)$ with parameters $(k,a,b) = (k,0,1)$, and initial conditions $A(1) = A(2) = 1$, satisfies $A(E_n) = E_{n-1}$ where
$E_n$ is defined by $E_n = E_{n-1} + E_{n-k}$ with $E_n = 1$ for $1 \leq n \leq k$.
\end{abstract}

\theoremstyle{plain}
\newtheorem{theorem}{Theorem}[section]
\newtheorem{lemma}[theorem]{Lemma}
\newtheorem{proposition}{Proposition}
\newtheorem{corollary}[theorem]{Corollary}

\numberwithin{equation}{section}
\numberwithin{theorem}{section}

\newcommand{\N}{\mathbb N}
\newcommand{\seqnum}[1]{\href{http://www.research.att.com/cgi-bin/access.cgi/as/~njas/sequences/eisA.cgi?Anum=#1}{\underline{#1}}}

\section{Introduction} \label{Sec:intro}

We study the behaviour of sequences defined by two types of recursions:
\begin{equation} \label{eqn:A}A(n) = A(n-a-A^k(n-b)) + A(A^k(n-b)) \end{equation}
\begin{equation} \label{eqn:C}C(n) = C(n-s-C(n-1)) + C(n-s-2-C(n-3)). \end{equation}
In recursion \eqref{eqn:A}, the parameters $a \geq 0, b \geq 1$ and $k \geq 1$ are integers and the superscript $k$ denotes a
$k$-fold composition of $A(n)$. This recursion generalizes one studied by Conway and others corresponding to $k=1, a= 0$
and $b=1$ \cite{Conolly, Mallows}. Grytczuk \cite{Grytczuk} studied one of these generalizations in detail.
Recursion \eqref{eqn:C} is a special case of recursions of the form
\begin{equation} \label{eqn:MF} C(n) = \displaystyle \sum_{i=1}^k C(n-a_i-C(n-b_i)) \end{equation}
where the parameters $a_i \geq 0$ and $b_i \geq 1$ are integers. Recursion \eqref{eqn:MF} with parameters
$k=2,\,a_1=0,\,b_1=1,\,a_2=1$ and $b_2=2$ is a well-known meta-Fibonacci recursion considered by Conolly and
others \cite{Conolly, Tanny2}. As such, recursions of the form \eqref{eqn:MF} are sometimes called Conolly type.
These recursions, in particular recursion \eqref{eqn:C} along with its variants, have received recent attention due to their rich
combinatorial properties under very specific sets of initial conditions (see \cite{Tanny, Ruskey} and the references cited therein).

Given recursions of the form \eqref{eqn:A} and \eqref{eqn:MF} along with some initial conditions, it is not immediate that such
recursions are well-defined for all $n \geq 1$ in the sense that for any $n$, arguments of the form $A^k(n-b), n-a-A^k(n-b)$ or
$n-a-C(n-b)$ lie in the interval $[1,n-1]$. The value of these arguments must necessarily lie within $[1,n-1]$ for the recursive
definition to work. A sequence of positive integers $\{a_n\}$ is called slow-growing if $a_n - a_{n-1} \in \{0,1\}$ for all $n$. In this
paper we derive properties of the initial conditions of \eqref{eqn:A} and \eqref{eqn:C} which guarantee that the recursions are
defined for all positive integers $n$, and that the resulting sequence is slow-growing. Slow-growing meta-Fibonacci sequences have
been the subject of much study (see, for example, \cite{Conolly, Tanny, Mallows, Tanny2} and the references cited therein).

We also consider sequences satisfying (\ref{eqn:A}) with parameters $(k,a,b) = (k, 0,1)$, which have been studied by Grytczuk
\cite{Grytczuk}. For $A(n)$ corresponding to the parameters $(k,a,b) = (2,0,1)$ with initial conditions $A(1) = A(2) = 1$, Grytczuk
found a correspondence between the resulting sequence $A(n)$ and certain operations on binary words. He used this method
to show that $A(F_n) = F_{n-1}$ where $F_n$ are the Fibonacci numbers, and stated that similar phenomenon should hold for $A(n)$
with parameters $(k,a,b) = (k,0,1)$ and initial conditions $A(1) = A(2) = 1$. Namely that $A(E_n) = E_{n-1}$ where $E_n$ is the generalized
Fibonacci sequence defined by $E_n = E_{n-1} + E_{n-k}$ with $E_n = 1$ for $1 \leq n \leq k$. Grytczuk stated that his methods could be
generalized to prove this property but we shall take an alternate - much simpler -  route to prove it.

\section{The Conway type recursion $A(n)$} \label{sec:A}

Consider recursion \eqref{eqn:A} with fixed parameters $(k,a,b)$. Suppose recursion \eqref{eqn:A} is given $b+j$
$(j \geq 0)$ slow-growing initial conditions that are positive integers. We shall say that $A(n)$ is slow-growing until term $m$ if
$A(n) - A(n-1) \in \{0,1\}$ for $n \in [2,m]$. For $n > b+j$ the computation of $A(n)$ requires that both the arguments $A^k(n-b)$ and
$n-a-A^k(n-b)$ in the recursive evaluation of $A(n)$ satisfy $0 < A^k(n-b) < n$ and $0 < n-a-A^k(n-b) < n$. As $a \geq 0$, these two
conditions for term $n$ are equivalent to \begin{equation} \label{N} 0 < A^k(n-b) < n-a \quad \text{for}\; n > b+j.
\end{equation}

We assume that $A(1) = 1$. For $A(b+j+1)$ to be defined, condition \eqref{N} for term $b+j+1$ requires that $A^k(j+1) \in (0, j+1 +b-a)$.
However, the slow-growing and positive initial conditions, along with the fact $A(1) = 1$, imply that $A^k(i)$ is positive and slow-growing
up to term $b+j$. Since $j+1 \leq b+j$, $A^k(j+1)$ lies within the initial conditions, and so the verification of condition \eqref{N} for term
$b+j+1$ depends on the initial conditions. Also, we require that $A(b+j+1) - A(b+j) \in \{0,1\}$ for slow-growth. In turns out that these conditions
are sufficient for $A(n)$ to be well-defined for all positive integers $n$ and be slow-growing. We prove this in the following proposition. But
note that if $b > a$ then $A(b+j+1)$ will be defined because the slow-growth of $A^k(i)$ up to term $b+j$ implies that $0 < A^k(j+1) \leq A^k(1) + j
= j+1 < (b-a) + j+1$, which establishes condition \eqref{N} for term $b+j+1$.

\begin{proposition} \label{thm:A}
Let $A(n) = A(n-a-A^k(n-b)) + A(A^k(n-b))$ with $a \geq 0, b\geq 1$, and $k \geq 1$ integers. Suppose that $A(n)$ is given
$b+j$ initial conditions ($j \geq 0$) that satisfy:
\begin{itemize}
\item[I)] The $b+j$ initial conditions are positive integers, slow-growing, and $A(1) = 1$.
\item[II)] $A(b+j+1)$ is defined and satisfies $A(b+j+1) - A(b+j) \in \{0,1\}$. 
\end{itemize}
Then $A(n)$ is defined for all positive integers $n$, remains slow-growing, and is unbounded. \end{proposition}

\begin{proof}We induct on $n$ to show that $A(n)$ is defined and slow-growing for $n \geq  b+j+1$. Set $\Delta(n) = A(n) - A(n-1)$. 
Hypotheses $II$ guarantees the existence of $A(b+j+1)$ and that $\Delta(b+j+1) \in \{0,1\}$, which starts off the induction process.

Assume that $A(i)$ is defined and slow-growing until term $n$. We first show that $A(n+1)$ is defined. For this we need
$A^k(n+1-b) \in (0,n+1-a)$ so that condition \eqref{N} is satisfied for $n+1$. The fact that $A(i)$ is slow-growing up to term $n$ and
$A(1) = 1$ imply that for $1 \leq i \leq n$, $A(i) \leq A(1) + i-1 = i$. Thus $A^2(i)$ is defined up to term $n$, satisfies $A^2(1) = 1$, and
remains slow-growing up to $n$ due to $A(i)$ being slow-growing till $n$. Iterating this argument, it follows that for any integer $l \geq 1$,
$A^l(i)$ is defined and slow-growing up to term $n$. As $b \geq 1$ and $n \geq b+j+1$, we have that $2 \leq n-b+1 \leq n$. Thus using
condition \eqref{N} for $n$ and the fact that $A^k(i)$ is slow-growing up to $n$, we deduce that $A^k(n+1-b) \leq A^k(n-b) +1
< n-a+1$ and $A^k(n+1-b) > 0$ since the initial conditions are positive. This shows that $A(n+1)$ is defined.

To verify the slow-growing property at $n+1$ we rearrange the terms of $\Delta(n+1)$ as
$$\Delta(n+1) = [A(A^k(n+1-b)) - A(A^k(n-b))] + [A(n+1-a-A^k(n+1-b)) - A(n-a-A^k(n-b))].$$
The two cases to consider are $\Delta(n+1-b) = 0$ or $\Delta(n+1-b) = 1$ for $n \geq b+j+1$ (note that the argument lies in the interval
$[2,n]$ so there is no problem with well-definedness).\\

\textbf{Case 1:} $\Delta(n+1-b) = 0$. In this case $A(n+1-b) = A(n-b)$, which implies that $A^k(n+1-b) = A^k(n-b)$. 
The compositions are defined as we noted that $A^k(i)$ is defined and slow-growing up to $n$. 
As $A^k(n+1-b) = A^k(n-b)$, the first summand in the rearranged version of $\Delta(n+1)$ vanishes, and after substituting $A^k(n-b)$ for 
$A^k(n+1-b)$ into the second summand we get
\begin{eqnarray*}
\Delta(n+1) &=& A(n+1-a-A^k(n-b)) - A(n-a-A^k(n-b))\\
&=& \Delta(n-a-A^k(n-b) + 1). \end{eqnarray*} 
Condition \eqref{N} for $n$ implies that $1 < n-a-A^k(n-b) + 1 < n+1$. This bound for the argument $n-a-A^k(n-b) + 1$ guarantees that
$\Delta(n-a-A^k(n-b) + 1)$ is defined as the argument lies in $[2,n]$. The induction hypothesis now implies that $\Delta(n+1) \in \{0,1\}$.\\

\textbf{Case 2:} $\Delta(n+1-b) = 1$. This implies that $A(n+1-b) = A(n-b) + 1$. Using the fact that $A^{k-1}(i)$ is slow-growing
up to $n$, and that $A(n-b) + 1 \leq n-b+1 \leq n$ due to $b \geq 1$, it follows that $A^k(n+1-b) = A^{k-1}(A(n-b)+1) = A^k(n-b) + \delta$
where $\delta \in \{0,1\}$. If $\delta = 0$ then the proof goes through the same route as case 1 because we get $A^k(n+1-b) = A^k(n-b)$
as above. If $\delta = 1$ then substituting $A^k(n+1-b)$ with the value $A^k(n-b) + 1$ gives 
\begin{eqnarray*} \Delta(n+1) &=& A(A^k(n-b)+1) - A(A^k(n-b)) \\
&+& A(n+1-a-A^k(n-b)-1) - A(n-a-A^k(n-b))\\
&=& A(A^k(n-b)+1) - A(A^k(n-b))\\
&=& \Delta(A^k(n+1-b)) \quad \text{since}\; A^k(n+1-b) - 1 = A^k(n-b). \end{eqnarray*}
The last quantity $\Delta(A^k(n+1-b)) \in \{0,1\}$ because $2 \leq A^k(n+1-b) \leq n$ by condition \eqref{N} for term $n$, and the fact
that $A^k(n+1-b) = A^k(n-b) + 1$.

Induction shows that $A(n)$ is defined and slow growing for all $n$. $A(n)$ must be unbounded, for otherwise, there is a $m$ 
such that $A(m+i) = M$ for all $i \geq 0$. Then for $i \geq b$, the definition of $A(n)$ and the boundedness assumption implies that
$M = A(m+i) = A(m+i-a-\hat{M}) + A(\hat{M})$ where $\hat{M} = A^{k-1}(M) \in [1,M]$. Set $i = \max\,\{b, \hat{M} + a\}$
to get that $M = A(\hat{M}) + M$; a contradiction. \end{proof}

When $A(n)$ is considered with parameters $(k,a,b)$ satisfying $b > a$, and initial conditions $A(1) = \cdots = A(b+j) = 1$ then it generates
a slow growing and unbounded sequence by Proposition \ref{thm:A}. The sequences generated by $A(n)$ with parameters $(k,0,1)$ and
initial  conditions $A(1) = A(2) = 1$ have been studied by Grytczuk as mentioned in the introduction. See Table \ref{tbl:A} for values of $A(n)$
with $k=2$. Fix the parameter $k$ and consider the corresponding sequence $E_n$ defined by $E_n = E_{n-1} + E_{n-k}$ with initial conditions
$E_n = 1$ for $1 \leq n \leq k$. We now prove the property of $A(n)$ which relates to $E_n$ as was mentioned in the introduction (see Table
\ref{tbl:A} to observe this phenomenon for $k=2$).

\begin{table}[!ht]
\fontsize{10}{10}\selectfont
\caption{\small$A(n)$ with parameters $(2,0,1)$ and $A(1) = A(2) = 1$.}\label{tbl:A}
\begin{tabular}{|r | r r r r r r r r r r r r r r r r r r r r r r|}
\hline
$n$&1&\textbf{2}&\textbf{3}&4&\textbf{5}&6&7&\textbf{8}&9&10&11&12&\textbf{13}&14&15&16&17&18&19&20&\textbf{21}&22\\\hline
$A(n)$&1&\textbf{1}&\textbf{2}&3&\textbf{3}&4&5&\textbf{5}&6&7&7&8&\textbf{8}&9&10&11&12&12&12&13&\textbf{13}&14\\\hline
\end{tabular}
\end{table}

\begin{theorem} \label{thm:B}
Let $A(n) = A(A^k(n-1)) + A(n-A^k(n-1))$ with $A(1) = A(2) = 1$. Then $A(E_n) = E_{n-1}$ for $n\geq 2$.\end{theorem}

\begin{proof}
It is immediate from the initial conditions of $E_n$ and of $A(i)$ that for $2 \leq j \leq k$, $A(E_j) = A(1) = 1 = E_{j-1}$. 
For the other cases we will use a three statement induction argument on the index $n$ of $E_n$. Throughout this
argument we use that $A(i)$ is slow growing and unbounded as per Proposition \ref{thm:A}. The hypotheses are
that for $n > k$, \begin{itemize}
\item[1.] $A(E_n) = E_{n-1}$
\item[2.] $A(E_{n-j} - 1) = E_{n-j - 1}$ for some $0 \leq j \leq k-1$
\item[3.] $A(E_n + 1) = E_{n-1} + 1$
\end{itemize}

An easy inductive argument or calculation shows that for $1 \leq j \leq k,\; E_{k+j} = j + 1$, and $A(j+1) = j$ (see Table \ref{tbl:A} for the $k=2$ case).
Hence, for $1 \leq j \leq k$, $A(E_{k+j}) = A(j+1) = j = E_{k+j - 1}$. For the base case $k+1$, we get $A(E_{k+1}) = A(2) = 1 = E_k,\;
A(E_{k+1} - 1) = A(1) = 1 = E_k$  and $A(E_{k+1} + 1) = A(3) = 2 = E_k + 1$. Assume that the three hypotheses hold up to $n$, and note since
$A(i)$ is slow and unbounded, $A^{-1}(\{E_n\}) = [\alpha, \beta]$. In the following, if $k=1$ then take $A^{k-1}(i) = i$.

\begin{eqnarray*}E_n + 1 = A(\beta + 1) &=& A(\beta + 1 - A^{k-1}(E_n)) + A(A^{k-1}(E_n))\\
&=& A(\beta + 1 - E_{n+1 - k}) + E_{n-k} \quad \text{(by hypothesis 1)}.\end{eqnarray*}

Therefore, $A(\beta + 1 - E_{n+1 - k}) = E_n - E_{n-k} + 1 = E_{n-1} + 1$. Hypothesis 1 and 3 for $n$, and the slow-growing property imply
that $A(i) > E_{n-1}$ if and only if $i > E_n$. So $\beta + 1 - E_{n+1 - k} > E_n$, and hence $\beta \geq E_n + E_{n+1-k} = E_{n+1}$.
Similarly, $A(\alpha - 1) = E_n - 1$ and\begin{equation*}E_n = A(\alpha) = A(\alpha - A^{k-1}(E_n - 1)) + A^k(E_n - 1).\end{equation*}

By using the first two hypotheses along with the fact that $A(i)$ is slow-growing, we claim that $A^k(E_n - 1) = E_{n-k}$.
Indeed, $A(E_n - 1) = E_{n-1}$ or $A(E_n - 1) = E_{n-1}-1$ by hypothesis 1 and the fact that $A(i)$ is slow-growing. In the former
case we use hypothesis 1 repeatedly to deduce that $A^k(E_n-1) = E_{n-k}$. In the latter case, we get $A^2(E_n-1) = A(E_{n-1} - 1)$
but $A(E_{n-1}-1) = E_{n-1}$ or $E_{n-2}-1$ by hypothesis 1 and the slow-growth of $A(i)$. If $A(E_{n-1}-1) = E_{n-2} - 1$ we keep
repeating the argument, and use hypothesis 2 for $n$ to eventually find a $j \in [0,k-1]$ such that $A(E_{n-j} - 1) = E_{n-j-1}$. At that
point we have the equation $A^{j+1}(E_n - 1) = E_{n-j-1}$ to which we compose $A(i)$ the remaining $k - j - 1$ times to deduce that
$A^k(E_n - 1) = E_{n-k}$ via hypothesis 1.

Analogously, we deduce that $A^{k-1}(E_n - 1)  = E_{n-k+1} - \delta$ where $\delta \in \{0,1\}$. Therefore $E_{n-1} = E_n - E_{n-k} =
A(\alpha - E_{n-k+1} + \delta)$, which implies that $\alpha - E_{n-k+1} + \delta \leq E_n$ because $A(i) > E_{n-1}$ for $i > E_n$. Thus
$\alpha \leq E_n + E_{n-k+1} = E_{n+1}$, and with $\beta \geq E_{n+1}$ from the previous paragraph, we deduce that $A(E_{n+1}) = E_n$.

In order to support hypothesis 2 for $n+1$, we note that if $A(E_{n-j} - 1) = E_{n-j - 1}$ for any $0 \leq j < k-1$ then there is nothing to show. 
So we can assume from hypothesis 2 for $n$ that $A(E_{n-k+1} - 1) = E_{n-k}$ and $A(E_{n-j} - 1) = E_{n-j-1} - 1$ for $0 \leq j < k-1$.
Aiming for a contradiction, if $A(E_{n+1} - 1) \neq E_n$ then it would follow that $A(E_{n+1} - 1) = E_n - 1$ because $A(E_{n+1}) = E_n$
as shown above, and $A(i)$ is slow-growing. Then from the assumption above for hypothesis 2, we get $A^k(E_{n+1} - 1) = E_{n-k+1} - 1$
after applying $A(i)$ to both sides of  $A(E_{n+1}-1) = E_n -1$ a total of $k-1$ times. But now \begin{eqnarray*}
E_n = A(E_{n+1}) &=& A(E_{n+1} - E_{n-k+1} + 1) + A(E_{n-k+1} - 1) \\
&=& A(E_n  + 1) + E_{n-k} \quad \text{(recall that $A(E_{n-k+1}-1) = E_{n-k}$).} \end{eqnarray*}
Therefore, $A(E_n + 1) = E_n - E_{n-k} = E_{n-1}$; a contradiction to hypothesis 3 for $n$.

To establish hypothesis 3 for $n+1$ we note that
\begin{eqnarray*}A(E_{n+1} + 1) &=& A(E_{n+1} + 1 - E_{n-k+1}) + E_{n-k} \quad \text{(by hypothesis 1)}\\
&=& A(E_n + 1) + E_{n-k} \\
&=& E_{n-1} + 1 + E_{n-k} = E_{n} + 1 \quad \text{(by hypothesis 3 for $n$).}\end{eqnarray*}
This completes the inductive argument.
\end{proof}

One implication of Theorem \ref{thm:B} is that if $\lim_{n \to \infty} \frac{A(n)}{n}$ exists then it must be the same as $\lim_{n \to \infty} \frac{E_{n-1}}{E_n}$.
The latter limit is $\phi_k^{-1}$ where $\phi_k$ is the largest positive root of the polynomial $x^k -x^{k-1} -1$: the characteristic polynomial of the
recursion $E_n = E_{n-1} + E_{n-k}$. However, it is an open question as to whether $\frac{A(n)}{n}$ converges as $n \to \infty$ for any $k \geq 2$.
For $k=1$, Mallows showed in \cite{Mallows} that $\lim_{n \to \infty} \frac{A(n)}{n} = \frac{1}{2}$, settling a question of Conway.

\section{The Conolly type recursion $C(n)$}

Consider the general Conolly type recursion \eqref{eqn:MF}. Fix a particular recursion from the family \eqref{eqn:MF} 
given with $r \geq \max\,\{b_i\}$ initial conditions that are positive integers. For the resulting sequence $C(n)$, define
$\Delta(n) = C(n) - C(n-1)$ and $\Delta_i(n) = C(n-a_i-C(n-b_i)) - C(n-1-a_i-C(n-1-b_i))$, so that $\Delta(n) = \sum_{i=1}^k
\Delta_i(n)$. The following lemma appears first in \cite{Tanny} as Lemma 6.1, where the authors deal with another recursion
of the form in \eqref{eqn:MF}.

\begin{lemma} \label{lem1}
Suppose that for $m > r$ the sequence $C(n)$ is defined up to term $m+1$ and slow-growing up to term $m$.
Then $\Delta_i(m+1) \in \{0,1\}$ for $1 \leq i \leq k$. When $k=2$, if $\Delta(m+1) \notin \{0,1\}$ then $\Delta(m+1) = 2$,
and $\Delta_1(m+1) = \Delta_2(m+1) = 1$. Also when $k=2$, for $r+1 < n \leq m+1$, $\Delta_i(n) = 1$ if and only if
$\Delta(n-b_i) = 0$ and $\Delta(n-a_i - C(n-1-b_i)) = 1$.\end{lemma}

\begin{proof}By the assumptions that $C(n)$ is slow-growing up to term $m$ and $b_i \geq 1$, it follows that $\Delta(m+1-b_i) \in \{0,1\}$.
If $\Delta(m+1-b_i) = 1$ then $$[m+1-a_i-C(m+1-b_i)] - [m-a_i-C(m-b_i)] = 1 - \Delta(m+1-b_i) = 0.$$ Thus $\Delta_i(m+1) = 0$ because
the arguments in both terms of the difference are equal. If $\Delta(m+1-b_i) = 0$ then $m+1-a_i-C(m+1-b_i) = 1 + (m-a_i-C(m-b_i))$,
and hence $\Delta_i(m+1) = \Delta(m+1-a_1-C(m-b_i))$. Note that $m+1-a_1-C(m-b_i) \in [2,m]$ since the existence of $C(m)$
requires that $C(m-b_i) \in (0,m-a_i)$. Thus $\Delta(m+1-a_1-C(m-b_i)) \in \{0,1\}$.

When $k =2$, since each $\Delta_i(m+1) \in \{0,1\}$, if $\Delta(m+1) \notin \{0,1\}$ then we must have each $\Delta_i(m+1) = 1$
and $\Delta(m+1) = \Delta_1(m+1) + \Delta_2(m+1) = 2$. Furthermore, the same calculations in the previous paragraph
with $m+1$ replaced with $n$ shows that for $r+1 < n \leq m+1$, $\Delta_i(n) = 1$ if and only if $\Delta(n-b_i) = 0$ and
$\Delta(n-a_i - C(n-1-b_i)) = 1$.
\end{proof}

Now we focus on the recursion $C(n) = C(n-s-C(n-1)) + C(n-s-2-C(n-3))$ in \eqref{eqn:C} given with $r \geq 3$ initial conditions that are
positive integers. For $n > r$, the term $C(n)$ is defined if and only if $C(n-1) \in (0,n-s)$ and $C(n-3) \in (0,n-s-2)$; in fact, the second
condition suffices. For notational convenience, let $C_1(n) = C(n-s-C(n-1))$ and $C_2(n) = C(n-2-s-C(n-2))$, and note that
$C_1(n-2) = C_2(n)$. Let $\Delta(n), \Delta_1(n)$ and $\Delta_2(n)$ be defined as before.

\begin{lemma} \label{lem2}
Suppose $C(n)$ is defined and slow-growing up until term $m > r$, and in addition there is no $n$ with $3 \leq n \leq m$
such that $\Delta(n) = \Delta(n-1) = 1$. Then for all $n$ with $r < n \leq m$, $C_1(n)-C_2(n) \in \{0,1\}$.\end{lemma}

\begin{proof}
By the assumption of slow-growth, $C(n-1) - C(n-3)\in \{0,1,2\}$. Thus the difference $d = [n-s-C(n-1)] - [n-s-2-C(n-3)]$ lies in $\{0,1,2\}$ as well.
We have that $$C_1(n) = C(n-s-C(n-1)) = C(n-s-2-C(n-3) + d).$$ If $d=0$ then $C_1(n) = C_2(n)$. If $d=1$ then $C_1(n) - C_2(n) =
\Delta(n-s-1-C(n-3)) \in \{0,1\}$ because $2 \leq n-s-1-C(n-3) \leq m$ where the first inequality follows since $C(n)$ is defined and the second
follows as $n \leq m,\,s\geq 0$ and $C(n-3) \geq 1$. Lastly, if $d=2$ then $C_1(n)-C_2(n) = \Delta(n-s-C(n-3)) + \Delta(n-s-C(n-3)-1)$.
Our assumption guarantees that there cannot be two consecutive differences of 1 since the argument of each $\Delta$ term is within $[2,m]$.
Thus we get $C_1(n)-C_2(n) \in \{0,1\}$.
\end{proof}

The consequence of Lemma \ref{lem2} is that so long as the assumptions are met up to $m > r$, for $ n \in [r+1,m]$, if $\Delta(n)$ is
even then $C_1(n) = C_2(n) = \frac{C(n)}{2}$. If $\Delta(n)$ is odd then $C_1(n) = C_2(n)+1 = \frac{C(n)+1}{2}$.

\begin{theorem} \label{thm:C}
Suppose that the recursion for $C(n)$ in \eqref{eqn:C} is given $r \geq 3$ initial conditions that are positive integers.
Suppose the following holds:
\begin{itemize}
\item[I)]The term $C(r+1)$ is defined via the recursion for $C(n)$.
\item[II)]There does not exists any $3 \leq n \leq r$ such that $\Delta(n) = \Delta(n-1) = 1$.
\item[III)]The sequence $C(n)$ is slow-growing up to term \,$r+1$.\footnote{Clearly if $C(n)$ is slow-growing up to term
$r+1$ then $C(r+1)$ must be defined. But we still keep hypothesis $I$ for clarity of exposition in the proof.}
\end{itemize}
Then $C(n)$ is defined and slow-growing for all $n \geq 1$. There also does not exists any $n \geq 3$ such that $\Delta(n) = \Delta(n-1) = 1$.
\end{theorem}

\begin{proof}
Suppose not for the sake of a contradiction. We consider three cases for how the claim could fail to be true,
based on which of the three conditions fails first.\\

\textbf{Case 1:} There is a minimal $m$ such that $C(m)$ is not defined while $C(n)$ is slow-growing until term $m-1$.
Then $m > r + 1$ by hypothesis $I$. The term $C(m)$ will be defined if and only if $1 \leq m-s-C(m-1) \leq m-1$ and
$1 \leq m-s-2-C(m-3) \leq m-1$. The second inequality for both cases is trivial since $C(m-1) \geq 1$. For the first 
inequalities, note that $m-1 \geq r+1$ so that $C(m-1)$ is defined via the recursion. As such $1 \leq m-1-s-C(m-2)$ and
$1 \leq m-s-3-C(m-4)$. By the assumption of slow-growth until $m-1$, we know that $\Delta(m-1)$ and $\Delta(m-3)$ lie in
$\{0,1\}$. Thus $C(m-1) \leq C(m-2) + 1 \leq m-1-s$ and $C(m-3) \leq m-s-3$ from the two inequalities in the previous sentence.
This establishes that $C(m)$ is indeed defined contrary to assumption.\\

\textbf{Case 2:} There is a minimal value of $m$ such that $\Delta(m)= \Delta(m-1)=1$, and for this minimal value of $m$, $C(n)$ is
slow-growing up to term $m$. Then $m > r$ by hypothesis $III$, and there is no $3 \leq n < m$ such that $\Delta(n) = \Delta(n-1) = 1$.
If $C(m)$ is even then by Lemma \ref{lem2}, $C_1(m) = C_2(m) = \frac{C(m)}{2}$. Our assumption implies that $C(m-2) = C(m) - 2$,
so $C(m-2)$ is also even and $C_1(m-2) = \frac{C(m)}{2} - 1$. Thus $C_2(m) \neq C_1(m-2)$, which is a contradiction to
$C_2(m) = C_1(m-2)$ as noted earlier. In the case that $C(m)$ is odd, Lemma \ref{lem2} implies that $C_1(m) = \frac{C(m)+1}{2}$.
Since $C(m-1)$ = $C(m)-1$, it follows that $C(m-1)$ is even and $C_1(m-1)=\frac{C(m)-1}{2}$. Hence $C_1(m) \neq C_1(m-1)$. But by using
$\Delta(m-1) = 1$ we get that $$C_1(m) = C(m-s-C(m-1)) = C(m-s-(C(m-2)+1)) = C_1(m-1).$$ This is another contradiction, and thus
it cannot be the case that $\Delta(m) = \Delta(m-1) = 1$.\\

\textbf{Case 3:} There is a minimal value of $m$ such that $\Delta(m) \notin \{0,1\}$, and for this minimal value of $m$,
it is not true that $\Delta(n) = \Delta(n-1) = 1$ for any $3 \leq n < m$. Then $m > r+1$ due to hypothesis $II$.
Lemma \ref{lem1} implies that $\Delta(m) = 2$, $\Delta_1(m) = 1$, and $\Delta_2(m) = 1$. Also by Lemma \ref{lem1}, $\Delta(m-1)=0$,
$\Delta(m-s-C(m-2)) = 1$, $\Delta(m-3)=0$,  $\Delta(m-s-2-C(m-4)) = 1$.

Now, consider the value of $\Delta(m-2)$, which by assumption lies in $\{0,1\}$. Assume for the sake of a contradiction that $\Delta(m-2) = 0$.
Under this assumption, we have that $\Delta_1(m-2) = 0$ since Lemma \ref{lem1} says that $\Delta_1(m-2), \Delta_2(m-2) \in \{0,1\}$ while
$\Delta(m-2) = \Delta_1(m-2) + \Delta_2(m-2)$. When $\Delta_1(m-2) = 0$, Lemma \ref{lem1} also implies that $\Delta(m-3)=1$ or
$\Delta(m-2-s-C(m-4)) = 0$. But this contradicts $\Delta(m-3)=0$ and $\Delta(m-s-2-C(m-4)) = 1$ from the previous paragraph, implying that
$\Delta(m-2) = 1$.

So we know that $\Delta(m-3) = 0$ and $\Delta(m-2)=1$. From this we deduce that the arguments $m-s-C(m-2)$ and $m-s-2-C(m-4)$ are
consecutive, while $\Delta(m-s-C(m-2)) = \Delta(m-s-2-C(m-4)) = 1$ from before. This means condition $II$ fails at $m-s-C(m-2)$, and
contradicts the minimality assumption provided $3 \leq m-s-C(m-2) < m$. The second inequality is trivial while for the first we note that since
$m > r + 1$, $C(m-1)$ is computed via the recursion. This implies that $C(m-4)  \leq m-4-s$, which is necessary for $C(m-1)$ to be defined.
Thus $C(m-2) = C(m-4) + 1 \leq m-3-s$ as required. \end{proof}

It is also clear that $C(n)$ must be unbounded despite not having consecutive increments. Indeed if $C(n)$ is bounded then
there exists a maximum value $M \geq 1$ and a $m \in \N$ such that $C(n) = M$ for all $n \geq m$, in light of the slow-growing
nature of $C(n)$. However, setting $N = m + 3 + s +M$, we see that $M = C(N) = C(N-s-M) + C(N-s-2-M) = C(m+3) + C(m+1) = 2M$;
a contradiction.\\

The upshot of Theorem \ref{thm:C} is that if the initial conditions of $C(n)$ are slow-growing and do not have consecutive increments
of 1, then $C(n)$ will be slow-growing for all $n$ as long as it is slow-growing until the term following the initial conditions.
For example, $C(n)$ with initial conditions all set to 1 has this property. So does $C(n)$ with initial conditions that give rise to a
combinatorial interpretation for the resulting sequence as explored in \cite{Ruskey, Tanny} and some of the references cited therein.
The following corollary concerns the behaviour of $\frac{C(n)}{n}$.

\begin{corollary} \label{cor1}
Let $C(n)$ satisfy all three hypotheses of Theorem \ref{thm:C}. Then $C(n)$ satisfies $\limsup_{n \to \infty} \frac{C(n)}{n} \leq \frac{1}{2}$
with equality unless $\liminf_{n \to \infty} \frac{C(n)}{n} = 0$.\end{corollary}

\begin{proof}Given $n \geq 2$, write $n -1 = 2q + r$ with $r \in \{0,1\}$. Since $C(n)$ is slow-growing and does not have 
consecutive increments, it follows that \begin{eqnarray*}
C(n) &=& C(1) + \sum_{i=1}^q [\Delta(2i) + \Delta(2i+1)] + \Delta(n)\cdot \delta_{r,1}\\
&\leq& C(1) + q+1 \quad \text{where}\; q = \lfloor \dfrac{n-1}{2} \rfloor\end{eqnarray*}
It follows that $\limsup_{n \to \infty} \dfrac{C(n)}{n} \leq 
\limsup_{n \to \infty} \dfrac{\lfloor \frac{n-1}{2} \rfloor}{n} = \dfrac{1}{2}$.
Using the recursive definition of $C(n)$ we get
\begin{equation*}\frac{C(n)}{n} = \frac{C(n-s-C(n-1))}{n-s-C(n-1)}\cdot \frac{n-s-C(n-1)}{n}
+ \frac{C(n-s-2-C(n-3))}{n-s-2-C(n-3)}\cdot \frac{n-s-2-C(n-3)}{n}\end{equation*}
Let $l = \liminf_{n \to \infty} \frac{C(n)}{n}$ and $u = \limsup_{n \to \infty} \frac{C(n)}{n}$. The bound
$C(n) \leq C(1) + \lfloor \frac{n-1}{2} \rfloor + 1$ shows that both $n-s-C(n-1)$
and $n-s-2-C(n-3)$ go to infinity as $n \rightarrow \infty$. Thus $\liminf_{n \to \infty} \frac{C(n-s-C(n-1))}{n-s-C(n-1)} \geq l$
and $\liminf_{n \to \infty} \frac{C(n-s-2-C(n-3))}{n-s-2-C(n-3)} \geq l$.
On the other hand, $$\liminf_{n \to \infty}\, \frac{n-s-C(n-1)}{n} = \liminf_{n \to \infty}\, \frac{n-s-2-C(n-3)}{n} = 1-u.$$
Thus after taking a liminf in the expression for $\frac{C(n)}{n}$ above, we get $l \geq 2l(1-u)$. Clearly $l \geq 0$, and so
$u \geq 1/2$ unless $l=0$. This establishes the corollary.\end{proof}

\section*{Further considerations}
In Theorem \ref{thm:C} we made the assumption that the initial conditions of $C(n)$ contain no consecutive increments along with
being slow-growing. It would be interesting to know whether the consecutive increments condition is necessary or simply
sufficient. The authors are not aware of any examples where the initial conditions have consecutive increments while $C(n)$ remains slow-growing.
Relating to Corollary \ref{cor1}, one question to consider is when does it hold that $\lim_{n \to \infty} \frac{C(n)}{n} = \frac{1}{2}$\,?

It would also be worthwhile to prove something similar to Theorem \ref{thm:C} for other Conolly type recursions of the form
\eqref{eqn:MF}. As far as the authors are aware almost all slow-growing Conolly type sequences result from a specific 
combinatorial interpretation of the corresponding recursion under sets of initial conditions that are forced on by the 
interpretation itself (see \cite{Tanny} and the references cited therein). The combinatorial interpretation does not consider the
case when all the initial conditions of the recursion under consideration are set to 1. So it would be interesting to explore what
other recursions of the form \eqref{eqn:MF} result in slow-growing sequences with all initial conditions equal to 1.

Going back to the recursion for $A(n)$ in \eqref{eqn:A} with parameters $(k,a,b) = (k,0,1)$ and initial conditions $A(1) = A(2) = 1$,
it would be of much interest to know whether $\lim_{n \to \infty} \frac{A(n)}{n}$ exists for all $k \geq 2$. As we stated, the value of
this limit - if it exists- follows from Theorem \ref{thm:B}. Finally, it would be worthwhile to study related recursions of the form
$$A(n) = A(n-a-A^k(n-b)) + A(A^k(n-c))$$ where $a \geq 0, b \geq 1$ and $c \geq 1$ are integers. Under what set of parameters
$k,a,b,c$ and initial conditions is the resulting sequence $A(n)$ defined for all positive integers and/or slow-growing?

\section*{Acknowledgements}
The authors would like to thank professor Steve Tanny for bringing the questions addressed in this paper to their attention.

\bigskip
\hrule
\bigskip

\noindent 2000 {\it Mathematics Subject Classification:} Primary 11B37. Secondary 11B39.

\noindent \emph{Keywords:}   meta-Fibonacci recursion; Conway sequence, Conolly sequence.

\bigskip
\hrule
\bigskip

\noindent (Concerned with sequences \seqnum{A093878}, \seqnum{A004001}, \seqnum{A008619}, \seqnum{A109964}.)

\end{document}